\documentclass[11pt]{amsart}

\usepackage{amsfonts,amssymb,amsmath,amsthm,latexsym,graphics,epsfig}
\usepackage{verbatim,enumerate,array,booktabs,ctable,color,bigstrut,prettyref}
\usepackage[all]{xy}
\usepackage[backref]{hyperref}
\usepackage[OT2,T1]{fontenc}

\newrefformat{eq}{\textup{(\ref{#1})}}
\newrefformat{prty}{\textup{(\ref{#1})}}

\definecolor{mylinkcolor}{rgb}{0.8,0,0}
\definecolor{myurlcolor}{rgb}{0,0,0.8}
\definecolor{mycitecolor}{rgb}{0,0,0.8}
\hypersetup{colorlinks=true,urlcolor=myurlcolor,citecolor=mycitecolor,linkcolor=mylinkcolor,linktoc=page,breaklinks=true}

\DeclareSymbolFont{cyrletters}{OT2}{wncyr}{m}{n}
\DeclareMathSymbol{\Sha}{\mathalpha}{cyrletters}{"58}

\newtheorem{defn}{Definition}[section]

\newtheorem{lemma}[defn]{Lemma}

\newtheorem{thm}[defn]{Theorem}
\newtheorem{theorem}[defn]{Theorem}
\newtheorem{cor}[defn]{Corollary}
\newtheorem{prop}[defn]{Proposition}

\newtheorem{conj}[defn]{Conjecture}

\theoremstyle{definition}

\newtheorem*{ack}{Acknowledgements}
\newtheorem{remark}[defn]{Remark}

\newtheorem{example}[defn]{Example}

\newcommand{\RR}{\mathbb R}
\newcommand{\QQ}{\mathbb Q}

\newcommand{\ZZ}{\mathbb Z}

\newcommand{\FF}{\mathbb F}

\newcommand{\OO}{\mathcal{O}}

\newcommand{\Gal}{\operatorname{Gal}}

\newcommand{\Cl}{\mathrm{Cl}}

\newcommand{\Sel}{{\rm Sel}}
\newcommand{\res}{{\rm res}}
\newcommand{\rank}{{\rm rank\,}}
\renewcommand{\Cl}{{\rm Cl}}
\newcommand{\Div}{{\rm Div}^{0}_-}
\newcommand{\val}{{\rm val}}

\renewcommand{\div}{{\rm div}}
\newcommand{\disc}{{\rm disc}}
\newcommand{\Sharp}{{\rm Sharp}}
\newcommand{\sgn}{{\rm sign}}

\begin{document}



\title[Rank bounds for jacobians]{Bounds of the rank of the Mordell--Weil group of Jacobians of Hyperelliptic Curves}

\author{Harris B. Daniels}
\address{Department of Mathematics and Statistics, Amherst College, Amherst, MA 01002, USA}
\email{hdaniels@amherst.edu}
\urladdr{\url{http://www3.amherst.edu/~hdaniels/}}

\author{\'Alvaro Lozano-Robledo}
\address{Department of Mathematics, University of Connecticut, Storrs, CT 06269, USA}
\email{alvaro.lozano-robledo@uconn.edu}
\urladdr{\url{http://alozano.clas.uconn.edu/}}

\author{Erik Wallace}
\address{Department of Mathematics, University of Connecticut, Storrs, CT 06269, USA}
\email{erik.wallace@uconn.edu}


\begin{abstract}
In this article we extend work of Shanks and Washington on cyclic extensions, and elliptic curves associated to the simplest cubic fields. In particular, we give families of examples of hyperelliptic curves $C: y^2=f(x)$ defined over $\QQ$, with $f(x)$ of degree $p$, where $p$ is a Sophie Germain prime, such that the rank of the Mordell--Weil group of the jacobian $J/\mathbb{Q}$ of $C$ is bounded by the genus of $C$ and the $2$-rank of the class group of the (cyclic) field defined by $f(x)$, and exhibit examples where this bound is sharp.
\end{abstract}
\maketitle
\section{Introduction}

Let $C/\QQ$ be a hyperelliptic curve, given by a model $y^2=f(x)$, with $f(x)\in\QQ[x]$, and let $J/\QQ$ be the jacobian variety attached to $C$. The Mordell--Weil theorem shows that $J(\QQ)$ is a finitely generated abelian group and, therefore,
$$J(\QQ)\cong J(\QQ)_\text{tors} \oplus \ZZ^{R_{J/\QQ}},$$
where $J(\QQ)_\text{tors}$ is the (finite) subgroup of torsion elements, and $R_{J/\QQ}=\operatorname{rank}_\ZZ(J(\QQ))\geq 0$ is the rank of $J(\QQ)$. In this article we are interested in bounds of $R_{J/\QQ}$ in terms of invariants of $C$ or $f(x)$. 

In \cite{Wash1}, Washington showed the following bound for the rank of certain elliptic curves, building on work of Shanks on the so-called simplest cubic fields (see \cite{Shanks1}).

\begin{theorem}[{\cite[Theorem 1]{Wash1}}]\label{thm-wash}
Let $m\geq 0$ be an integer such that  $m^2+3m+9$ is square-free. Let $E_m$ be the elliptic curve given by the Weierstrass equation
$$E_m\colon y^2 = f_m(x)= x^3+mx^2-(m+3)x+1.$$
Let $L_m$ be the number field generated by a root of $f_m(x)$, let $\operatorname{Cl}(L_m)$ be its class group, and let $\Cl(L_m)[2]$ be the $2$-torsion subgroup of $\Cl(L_m)$. Then,
$$\operatorname{rank}_\ZZ(E_m(\QQ)) \leq 1 + \operatorname{dim}_{\FF_2}(\operatorname{Cl}(L_m)[2]).$$
\end{theorem}

In this article, we extend Washington's result to curves of genus $g\geq 2$. In order to find other families of hyperelliptic curves of genus $g\geq 2$ where a similar bound applies, we use a method of $2$-descent for jacobians described by Cassels, Poonen, Schaefer, and Stoll (see Section \ref{sec-2descent}; in particular, we follow the implementation described in \cite{stoll}). The examples come from hyperelliptic curves $y^2=f(x)$ such that $f(x)$ defines the maximal real subfield of a cyclotomic extension of $\QQ$, and the degree of $f(x)$ is $p$, a Sophie Germain prime. We obtain the following result.

\begin{thm}\label{thm-sophie-intro}
Let $q\geq 7$ be a prime such that $p=(q-1)/2$ is also prime, and let $L=\QQ(\zeta_q)^+$ be the maximal totally real subfield of $\QQ(\zeta_q)$. Let $f(x)$ be the minimal polynomial of $\zeta_q+\zeta_q^{-1}$ or $-(\zeta_q+\zeta_q^{-1})$, let $C/\QQ$ be the hyperelliptic curve $y^2=f(x)$, of genus $g=(p-1)/2$, and let $J/\QQ$ be its jacobian. Then, there are constants $\rho_\infty$ and $j_\infty$, that depend on $q$, such that  
\begin{align*}
\rank_\ZZ (J(\QQ))\leq \dim_{\FF_2}\Sel^{(2)}(\QQ,J)  \leq \rho_{\infty} + j_\infty + \dim_{\FF_2} (\Cl^+(L)[2]),
\end{align*}
where $\rho_{\infty}+j_\infty\leq p-1$. Further, if one of the following conditions is satisfied, 
\begin{enumerate}
\item the Davis--Taussky conjecture holds (Conjecture \ref{conj-taussky}), or
\item the prime $2$ is inert in the extension $\QQ(\zeta_p)^+/\QQ$, or
\item  $q\leq 92459$,
\end{enumerate} then $\rho_{\infty}=0$ and $j_\infty=g=(p-1)/2$, and $\dim_{\FF_2}\Sel^{(2)}(\QQ,J) \leq  g + \dim_{\FF_2}(\Cl(L)[2]).$
\end{thm}

In fact, if the Davis--Taussky conjecture holds (see Remark \ref{rem-davistausskyoddclassnumber}), then the bound of Theorem \ref{thm-sophie-intro} becomes $\dim_{\FF_2}\Sel^{(2)}(\QQ,J) \leq  g$.

The organization of the paper is as follows. In Section \ref{sec-2descent}, we review the method of $2$-descent as implemented by Stoll in \cite{stoll}. In Sections \ref{sec-betterbounds}, \ref{sec-totally}, and \ref{sec-muchbetterbounds}, we specialize the $2$-descent method for the situations we encounter in the rest of the paper, namely the case when $f(x)$ defines a totally real extensions, or cyclic extension of $\QQ$, of prime degree. In Section \ref{sec-genus1}, we give a new proof of Washington's theorem using the method of $2$-descent. In Section \ref{sec-sophie} we provide examples of hyperelliptic curves of genus $g=(p-1)/2$ where $p$ is a Sophie Germain prime, and prove Theorem \ref{thm-sophie-intro}. Finally, in Section \ref{sec-examples}, we illustrate the previous sections with examples of curves, jacobians, and how their ranks compare to the bounds.

\begin{ack}
The authors would like to thank Keith Conrad, G\"urkan Dogan, Franz Lemmermeyer, Paul Pollack, and Barry Smith, for several helpful comments and suggestions. We would also like to express our gratitude to David Dummit for very useful suggestions and noticing an error in an earlier version of the paper. Finally we would like to express our thanks to the referees who have given us helpful suggestions and pointed out some errors in previous versions of this paper.
\end{ack}

\section{Stoll's Implementation of 2-Descent}\label{sec-2descent}
In this section we summarize the method of $2$-descent as implemented by Stoll in \cite{stoll}. The method was first described by Cassels \cite{cassels1}, and by Schaefer \cite{schaefer1}, and Poonen-Schaefer \cite{schaefer2} in more generality. Throughout the rest of this section we will focus on computing the dimension of the 2-Selmer group of the jacobian $J$ of a hyperelliptic curve $C$, given by an affine equation of the form
$$C:y^2=f(x),$$
where $f\in\ZZ[x]$ is square-free and $\deg(f)$ is odd (Stoll also treats the case when $\deg(f)$ is even, but we do not need it for our purposes). In this case, the curve $C$ is of genus $g=(\deg(f) - 1)/2$ with a single point at infinity in the projective closure. Before we can compute the dimension of the 2-Selmer group, we must define a few objects of interest and examine some of their properties. We will follow the notation laid out in \cite{stoll}.

Let $\Sel^{(2)}(\QQ,J)$ be the 2-Selmer group of $J$ over $\QQ$, and let $\Sha(\QQ,J)[2]$ be the $2$-torsion of the Tate-Shafarevich group of $J$ (as defined, for instance, in Section 1 of \cite{stoll}). Selmer and Sha fit in the following fundamental short exact sequence:
$$\xymatrix{
0\ar[r]& J(\QQ)/2J(\QQ)\ar[r]& \Sel^{(2)}(\QQ,J)\ar[r]& \Sha(\QQ,J)[2]\ar[r]& 0.
}$$
With this sequence in hand we get a relationship between for the rank of $J(\QQ)$ and the $\FF_2$-dimensions of the other groups that we defined.
\begin{equation}\label{eq-rank}
\rank_\ZZ J(\QQ) + \dim_{\FF_2} J(\QQ)[2] + \dim_{\FF_2}\Sha(\QQ,J)[2] = \dim_{\FF_2}\Sel^{(2)}(\QQ,J).
\end{equation}
Using equation (\ref{eq-rank}), we get our first upper bound on the rank
\begin{equation}\label{eq-bound}
\rank_\ZZ J(\QQ)\leq \dim_{\FF_2} \Sel^{(2)}(\QQ,J)-\dim_{\FF_2} \Sha(\QQ,J)[2]\leq \dim_{\FF_2} \Sel^{(2)}(\QQ,J).
\end{equation}
This upper bound is {\it computable}, in the sense that $J(\QQ)[2]$ and the Selmer group are computable, as we describe below.

\begin{defn}\label{defn-L}
For any field extension $K$ of $\QQ$ and $f\in \QQ[x]$, let $L_K = K[T] / (f(T))$ denote the algebra defined by $f$ and let $N_K$ denote the norm map from $L_K$ down to $K$.
\end{defn}

We denote $L_K=K[\theta]$, where $\theta$ is the image of $T$ under the reduction map $K[T]\to K[T]/(f(T))$, and $L_K$ is a product of finite extensions of $K$:
$$L_K=L_{K,1}\times\cdots\times L_{K,m_K},$$
where $m_K$ is the number of irreducible factors of $f(x)$ in $K[x]$. Here, the fields $L_{K,j}$ correspond to the irreducible factors of $f(x)$ in $K[x]$, and the map $N_K:L_K\to K$ is just the product of the norms on each component of $L_K$. That is, if $\alpha=(\alpha_1,\alpha_2,\dots,\alpha_{m_K})$, then $N_K(\alpha)=\prod_{i=1}^{m_K}N_{L_{K,i}/K}(\alpha_i)$ where $N_{L_{K,i}/K}:L_{K,i}\to K$ is the usual norm map for the extension of fields $L_{K,i}/K$.

In order to ease notation, we establish the following notational conventions: when $K=\QQ$ we will drop the field from the subscripts altogether, and if $K=\QQ_v$, we will just use the subscript $v$. This convention will apply to \emph{anything} that has a field as a subscript throughout the rest of the paper. As an example, $L_v=\QQ_v[T]/(f(T))$ and $L=\QQ[T]/(f(T))$.

Following standard notational conventions, we let $\OO_K,\ I(K),$ and $\Cl(K)$ denote the ring of integers of $K$, the group of fractional ideals in $K$, and the ideal class group of $K$, respectively. We  define analogous objects for the algebra $L_K$ as products of each component, as follows:
\begin{align*}
\OO_{L_K}&=\OO_{L_{K,1}}\times\cdots\times\OO_{L_{K,m_K}},\\
I(L_K)&=I({L_{K,1}})\times\cdots\times I({L_{K,m_K}}),\\
\Cl(L_K)&=\Cl({L_{K,1}})\times\cdots\times \Cl({L_{K,m_K}}).
\end{align*}

\begin{defn} \label{defn-res}
Let $K$ be a field extension of $\QQ$, and let $L=\QQ[T]/(f(T))$ be as before.
\begin{enumerate} 
\item Let $I_v(L)$ denote the subgroup of $I(L)$ generated, in each component, by fractional ideals in $L_{\QQ,i}$ with support above a prime $v$ in $\ZZ$. For a finite set $S$ of places of $\QQ$, let
$$I_S(L)=\prod_{v\in S\setminus\{\infty\}} I_v(L).$$

\item For any field extension $K$ of $\QQ$, let 
$$H_K = \ker\left( N_K\colon L_K^\times/(L_K^\times)^2\to K^\times/(K^\times)^2 \right).$$
For any place $v$ of $\QQ$, we let $\res_v\colon H\to H_v$ be the canonical restriction map induced by the natural inclusion of fields $\QQ\hookrightarrow \QQ_v$.
\item 
Let $\Div(C)$ denote the group of degree-zero divisors on $C$ with support disjoint from the principal divisor $\div(y)$.
\end{enumerate}
\end{defn}

\begin{remark}
In our case, the curve is given by $C:y^2=f(x)$, and the support of $\div(y)$ is exactly the  points with coordinates $(\alpha,0)$, where $\alpha$ is a root of $f$, and the unique point at infinity. 
\end{remark}
Now for each $K$, there is a homomorphism  
\[
F_K\colon \Div(C)(K)\to L^\times_K, \quad \sum_{P} n_PP\mapsto \prod_P (x(P) - \theta )^{n_P},
\]
and this homomorphism induces a homomorphism $\delta_K\colon J(K) \to H_K$ with kernel $2J(K)$ by \cite[Lemma 4.1]{stoll}. By abusing notation, we also use $\delta_K$ to denote the induced map $J(K)/2J(K) \to H_K$. 

All of these facts, together with some category theory, give us the following characterization of the 2-Selmer group of $J$ over $\QQ$. 
\begin{prop}[{\cite[Prop. ~4.2]{stoll}}]\label{prop-selmer_characterization} 
The $2$-Selmer group of $J$ over $\QQ$ can be identified as follows:
\[
\Sel^{(2)}(\QQ,J) = \{\xi\in H \mid \res_v(\xi) \in\delta_v(J(\QQ_v))\hbox{ for all } v\}.
\]
\end{prop}

In order to take advantage of this description of the Selmer group, we need some additional facts about the 2-torsion of $J$ and the $\delta_K$ maps.

\begin{lemma}[{\cite[Lemma 4.3]{stoll}}]\label{lem-deltak_maps} Let $K$ be a field extension of $\QQ$.
\begin{enumerate}
\item[\rm (1)] For a point $P\in C(K)$ not in the support of $\div(y)$, $\delta_K(P-\infty) = x(P) - \theta \bmod (L_K^\times)^2.$
\item[\rm (2)] Let $f = f_1\cdots f_{m_K}$ be the factorization of $f$ over $K$ into monic irreducible factors. Then, to every factor $f_j$, we can associate an element $P_j\in J(K)[2]$ such that:
\begin{enumerate}
\item[\rm (i)] The points $\{P_j\}$ generate $J(K)[2]$ and satisfy $\sum_{j=1}^{m_K}P_j = 0.$
\item[\rm (ii)] Let $h_j$ be the polynomial such that $f = f_jh_j$. Then 
$$\delta_K(P_j) = (-1)^{\deg(f_j)}f_j(\theta) + (-1)^{\deg(h_j)}h_j(\theta) \bmod (L_K^\times)^2.$$
\end{enumerate}
\item[\rm (3)] $\dim J(K)[2] = m_K -1.$

\end{enumerate}
\end{lemma}

\begin{defn}
Let $I_v = \ker(N\colon I_v(L)/I_v(L)^2 \to I(\QQ)/I(\QQ)^2)$ and let $\val_v\colon H_v\to I_v$ be the map induced by the valuations on each component of $L_v$. Considering all primes at once, we get another map $\val\colon H \to I(L)/I(L)^2$. More specifically, the $\val$ map is the product of $\val_v(\res_v)$ over all places $v$. 
\end{defn}

Next, the following lemma helps us compute the dimensions of various groups when $K$ is a local field or $\RR$.

\begin{lemma}[{\cite[Lemma 4.4]{stoll}}]\label{lem-local_dims}
Let $K$ be a $v$-adic local field, and let $d_K=[K:\QQ_2]$ if $v=2$ and $d_K = 0$ if $v$ is odd. Then
\begin{enumerate}
\item[\rm (1)] $\dim J(K)/2J(K) = \dim J(K)[2] + d_Kg =m_K-1+d_Kg.$
\item[\rm (2)] $\dim H_K = 2\dim J(K)/2J(K) = 2(m_K-1+d_Kg)$.
\item[\rm (3)] $\dim I_K = m_K - 1.$
\end{enumerate}
\end{lemma}

With all of this machinery the description of $\Sel^{(2)}(\QQ,J)$ given in Proposition \ref{prop-selmer_characterization} can be refined as follows.

\begin{prop}[{\cite[Cor. ~4.7]{stoll}}]\label{prop-selmer_recharacterization}
Let $S = \{\infty,2\}\cup\{v : v^2 \hbox{ divides } \disc(f) \}$. Then 
\[
\Sel^{(2)}(\QQ,J) = \{ \xi\in H \mid \val(\xi)\in I_S(L)/I_S(L)^2,\  \res_v(\xi)\in \delta_v(J(\QQ_v))\hbox{ for all } v\in S\}.
\]
\end{prop}
This new characterization suggests the following method to compute $\Sel^{(2)}(\QQ,J)$:

\begin{itemize}
\item[\bf S1:] Find the set $S$.
\item[\bf S2:] For each $v\in S$, determine $J_v = \delta_v(J(\QQ_v))\subseteq H_v.$
\item[\bf S3:] Find a basis for a suitable finite subgroup $\widetilde{H} \subseteq L^\times/(L^\times)^2$ such that $\Sel^{(2)}(\QQ,J)\subseteq \widetilde{H}.$
\item[\bf S4:] Compute $\Sel^{(2)}(\QQ,J)$ as the inverse image of $\prod_{v\in S} J_v$ under \[\prod_{v\in S} \res_v\colon \widetilde{H}\to \prod_{v\in S} H_v.\]
\end{itemize}

Ignoring any complications that arise from computing and factoring the discriminant of $f$, we focus on steps 2 and 3. We omit the details of how to carry out step 4, since we are only interested in an upper bound for the $\FF_2$-dimension of $\Sel^{(2)}(J,\QQ)$. Step 2 can be broken down into three substeps:
\begin{itemize}
\item[\bf S2.1:] For all $v\in S\setminus\{\infty\}$, compute $J_v = \delta_v(J(\QQ_v))$ and its image $G_v = \val_v(J_v)$ in $I_v$.
\item[\bf S2.2:] If $G_v = 0$ for some $v$, with $v$ odd, remove $v$ from $S$. 
\item[\bf S2.3:] Compute $J_\infty$. 
\end{itemize}

To complete step 2.3, we need the following lemma.

\begin{lemma}[{\cite[Lemma 4.8]{stoll}}]\label{lem-Real_dims} With notation as above:
\begin{enumerate}
\item[\rm (1)] $\dim J(\RR)/2J(\RR) = m_\infty-1 - g.$
\item[\rm (2)] $J_\infty$ is generated by $\{\delta_\infty(P-\infty) : P\in C(\RR)\} $.
\item[\rm (3)] The value of $\delta_\infty(P-\infty)$ only depends on the connected component of $C(\RR)$ containing $P$. 
\end{enumerate}
\end{lemma}

Next, for step 3, we see that if we let 
$$ G = \prod_{v\in S\setminus\{ \infty \}} G_v \subseteq I(L)/I(L)^2,$$ then the group 
$\{\xi\in H : \val(\xi)\in G\}$
contains $\Sel^{(2)}(\QQ,J)$. In fact, the larger group $\widetilde{H} = \{\xi\in L^\times/(L^\times)^2 :  \val(\xi)\in G \}$ also contains the 2-Selmer group and we can compute its basis using the following two steps.
\begin{itemize}
\item[\bf S3.1:] Find a basis of $V = \ker(\val\colon L^\times/(L^\times)^2 \to I(L)/I(L)^2)$.
\item[\bf S3.2:] Enlarge this basis to get a basis of $\widetilde{H} = \val^{-1}(G)$. 
\end{itemize}
With notation as above, Stoll deduces an upper bound and a formula for the $\FF_2$-dimension of the 2-Selmer group (see Lemma 4.10 and the discussion under Step 4), as follows.
\begin{prop}[{\cite[Lemma 4.10]{stoll}}]\label{prop-Sel2_UB}\label{stollbnd} With notation as above,
\[
\dim_{\FF_2}\Sel^{(2)}(\QQ,J)\leq (m_\infty-1) + \dim_{\FF_2}(\Cl(L)[2]) + \dim_{\FF_2}\ker\big(G\to\Cl(L)/2\Cl(L)\big).
\]
\end{prop}

In the next section, we modify the proof of the bound in Proposition \ref{prop-Sel2_UB} to allow for extra conditions at infinity, before we specialize to totally real, and  cyclic extensions.

\subsection{About the proof of Proposition \ref{prop-Sel2_UB}}\label{sec-betterbounds}

The following commutative diagrams helps understand where the Selmer group fits:
$$\xymatrix{
J(\QQ)/2J(\QQ)\ar@{^{(}->}[rrr]^{\delta} \ar[d] & & & H\ar[rrr]^{\val}\ar[d]^{\prod_v{\rm res}_v}& & &I(L)/I(L)^2 \ar[d] \\
\prod_v J(\QQ_v)/2J(\QQ_v)\ar@{^{(}->}[rrr]^{\prod_v \delta_v} & & &\prod_v H_v\ar[rrr]^{\prod_v \val_v} & & &\prod_v I_v(L)/I_v(L)^2 
} $$
The $2$-Selmer group of $J$ over $\QQ$ is then given, as in Prop. \ref{prop-selmer_recharacterization}, by
\[
\Sel^{(2)}(\QQ,J) = \{ \xi\in H : \val(\xi)\in I_S(L)/I_S(L)^2,\  \res_v(\xi)\in \delta_v(J(\QQ_v))\hbox{ for all } v\in S\}.
\]
The Selmer group is thus contained in $H\subseteq L^\times/(L^\times)^2$, and more precisely,
$$\Sel^{(2)}(\QQ,J)\subseteq \{ \xi \in H : \val(\xi) \in G,\ \text{res}_\infty(\xi) \in J_\infty \}$$
where $J_\infty=\delta_\infty(J(\RR))$, the group $ G$ is the product $ \prod_{v\in S\setminus\{ \infty \}} G_v \subseteq I(L)/I(L)^2,$ and recall that $H$ is the kernel of the norm map from $L^\times/(L^\times)^2$ down to $\QQ^\times/(\QQ^\times)^2$. Thus, $\Sel^{(2)}(\QQ,J)$ is contained in the larger group
$$\widehat{H}= \{ \xi \in L^\times/(L^\times)^2 : \val(\xi) \in G,\ \text{res}_\infty(\xi) \in J_\infty \}.$$
We emphasize here that the definition of $\widetilde{H}$ in \cite{stoll} does not impose a condition at $\infty$, but the definition of $\widehat{H}$ does to improve the bounds accuracy (thus $\widehat{H}\subseteq \widetilde{H}$). In an attempt to simplify notation, let $L_{J_\infty}$ be the subspace of $L^\times/(L^\times)^2$ with a condition added at infinite primes by $L_{J_\infty} = L^\times/(L^\times)^2 \cap \res_\infty^{-1}(J_\infty)$. Thus, $\widehat{H}=\{ \xi \in L_{J_\infty} : \text{val}(\xi) \in G\},$
and $\Sel^{(2)}(\QQ,J)\subseteq \widehat{H}$. 
\begin{remark} Note that $\widehat{H}$ is the largest subgroup of $L_{J_\infty}$ such that 
	$$\val(\widehat{H})\cong  G\cap \text{val}\left(L_{J_\infty}\right).$$
Let us show that indeed $\val(\widehat{H})\cong  G\cap \text{val}\left(L_{J_\infty}\right)$. Indeed:
\begin{itemize}
\item Suppose $\xi\in \widehat{H}$ and consider $\text{val}(\xi)$. By definition, since  $\xi\in \widehat{H}$, we have that $\text{val}(\xi)$ is in $G$, and $\xi \in L_{J_\infty}$, thus $\text{val}(\xi)\in \text{val}(L_{J_\infty})$. Hence, $\val(\xi)\in G\cap   \text{val}\left(L_{J_\infty}\right)$.
\item Conversely, suppose $g \in G\cap \text{val}\left(L_{J_\infty}\right)$. Then, there is some $\xi \in L_{J_\infty}$ such that $\text{val}(\xi)=g$. In particular, $\text{res}_\infty(\xi) \in J_\infty$ and since $\text{val}(\xi)=g\in G$, it follows that $\xi \in \widehat{H}$. Hence, $\text{val}(\xi)\in \text{val}(\widehat{H})$.
\end{itemize}
Also, let us show that $\widehat{H}$ is the largest subgroup of $L_{J_\infty}$ such that 
$\val(\widehat{H})\cong  G\cap \text{val}\left(L_{J_\infty}\right).$ Suppose that $\xi \in L_{J_\infty}$ and $\text{val}(\xi) \in G\cap \text{val}\left(L_{J_\infty}\right)$. Then, $\xi \in L_{J_\infty}$ and $\text{val}(\xi)\in G$, so $\xi \in \widehat{H}$ by definition.
\end{remark}

Next, we define subspaces $V$ and $W$ of $L^\times/(L^\times)^2$ as follows:
\begin{itemize}
\item Let $\{\xi_i\}_{i=1}^r$ be generators of $G\cap \text{val}\left(L_{J_\infty}\right)$ with $\dim(G\cap \text{val}\left(L_{J_\infty}\right))=r$, and for each $1\leq i\leq r$ pick one $\mu_i\in L_{J_\infty}$ such that $\val(\mu_i)=\xi_i$. Let $W$ be the subspace generated by $\{\mu_i\}_{i=1}^r$. Note that $W\subseteq L_{J_\infty} \subseteq L^\times/(L^\times)^2$. In particular, $\res_\infty(w)\in J_\infty$ for all $w\in W$. Moreover, $W$ and $\val(W)$ are isomorphic by construction, so
$$W\cong \val(W) = G\cap \text{val}\left(L_{J_\infty}\right) \subseteq G\cap \text{val}\left(L^\times/(L^\times)^2\right)=\ker(G\to \Cl(L)/2\Cl(L)).$$
Thus, $r=\dim(W)=\dim(G\cap \text{val}\left(L_{J_\infty}\right))\leq  \dim \ker(G\to \Cl(L)/2\Cl(L)).$

\item Next, let us write $V=\ker(\val\colon L_{J_\infty} \to I(L)/I(L)^2)$. It follows that $\widehat{H} = V \oplus W$ (note that $\val(V)$ is trivial, while $\val(w)$ is non-trivial for every $w\neq 0$ in $W$).
\end{itemize}

\begin{lemma}
 Let $V=\ker(\val\colon L_{J_\infty} \to I(L)/I(L)^2)$, let $U=\left(\OO_L^\times/(\OO_L^\times)^2\right)\cap L_{J_\infty}$, and let $\Cl(L_{J_\infty})$ be the class group defined as follows $ \Cl(L_{J_\infty})=I(L)/P(L_{J_\infty})$, where $I(L)$ is the group of fractional ideals of $L$, and $P(L_{J_\infty})$ is the group of principal fractional ideals $\mathfrak{A}=(\alpha)$ with a generator such that $\res_\infty(\alpha)\in J_\infty$.  Then, there is an exact sequence:
 $$0 \mapsto U  \to V \to \Cl(L_{J_\infty})[2] \to 0.$$
\end{lemma} 
\begin{proof}
Consider 
$$\xymatrix{
& L^\times \cap \res_\infty^{-1}(J_\infty) \ar[rr]^2\ar[d]& & L^\times \cap \res_\infty^{-1}(J_\infty)\ar[rr]\ar[d] & & L_{J_\infty}\ar[rr]\ar[d]^\val & &0  \\
0 \ar[r] & I(L)\ar[rr]^2 & & I(L)\ar[rr] & &I(L)/I(L)^2\ar[rr]  & &0
} 
$$
and apply the snake lemma.
\end{proof}

\begin{remark}
Let $P(L)$ be the subgroup of all principal fractional ideals, let $P^+(L)$ be the subgroup of principal ideals generated by a totally positive element, and let $P(L_{J_\infty})$ be as above. Since the trivial signature $(1,1,\ldots,1)\in J_\infty$, it implies that $P^+(L) \subseteq P(L_{J_\infty}) \subseteq P(L)$, and therefore there are surjections
$$\Cl(L) \twoheadleftarrow \Cl(L_{J_\infty}) \twoheadleftarrow \Cl^+(L),$$
where $\Cl^+(L)=I(L)/P^+(L)$ is the narrow class group of $L$.
\end{remark}

Putting all this together (and writing $\dim$ for $\dim_{\FF_2}$), we obtain a bound
\begin{align}\label{newbound}
\nonumber \dim\Sel^{(2)}(\QQ,J)& \leq \dim(\widehat{H}) = \dim U + \dim \Cl(L_{J_\infty})[2] + \dim W \\ 
\nonumber & = \dim \left(\OO_L^\times/(\OO_L^\times)^2\cap \res_\infty^{-1}(J_\infty)\right) + \dim \Cl(L_{J_\infty})[2] + \dim G\cap \text{val}\left(L_{J_\infty}\right)\\
 & \leq \dim \left(\OO_L^\times/(\OO_L^\times)^2\cap \res_\infty^{-1}(J_\infty)\right)  + \dim \Cl^+(L)[2] + \dim G\cap \text{val}\left(L_{J_\infty}\right).
\end{align}
We note here that $\dim \left(\OO_L^\times/(\OO_L^\times)^2\cap \res_\infty^{-1}(J_\infty)\right)\leq m_\infty -1$, where we have used the fact that $\dim(\OO_L^\times/(\OO_L^\times)^2)=m_\infty$, and the fact that $\res_\infty(-1)$ is not in $J_\infty$ (because $J_\infty=\delta_\infty(J(\RR))\subseteq H_\infty$, which is the kernel of the norm map, so $N(j)=1$ for $j\in J_\infty$, but the norm $N(-1)=-1$ because the degree of $L$ is odd). We will improve on the bound given by (\ref{newbound}) above by making certain assumptions about $G$ and a more careful analysis of the dimension of the subgroup of totally positive units. Before we state our refinements, we review some of the results on totally positive units that we shall need.

It is worth pointing out that the third line of Eq. (\ref{newbound}) is not necessarily an improvement over the bound in Prop. \ref{stollbnd} if $\dim \Cl^+(L)[2]> \dim \Cl(L)[2]$. In our setting, we will seek conditions where $\dim \Cl^+(L)[2]= \dim \Cl(L)[2]$ and then our bound in Eq. (\ref{newbound}) will be an improvement due to the more careful counting of units according to their infinite valuations.

\subsection{Totally Positive Units}\label{sec-totally}
Let $L$ be a totally real Galois number field of prime degree $p>2$, with embeddings $\tau_i : L \to \RR$, for $i=1,\ldots,p$, and maximal order $\OO_L$. Let $\Cl(L) = \Cl(\OO_L)$ be the ideal class group of $L$, and let $\Cl^+(L)$ be the narrow class group. Let $V_\infty = \{\pm 1\}^p\cong (\FF_2)^p$ and, by abuse of notation, we extend the map $\res_\infty$ (as in Definition \ref{defn-res}, where we note that $H_\infty\cong V_\infty$ and $H\subseteq L^\times/(L^\times)^2$) to
$$\res_\infty: L^\times/(L^\times)^2 \to V_\infty$$
by $\res_\infty(\alpha) = (\sgn(\tau_1(\alpha)),\sgn(\tau_2(\alpha)),\ldots,\sgn(\tau_p(\alpha)))$. Let $\OO_L^\times$ be the unit group of $\OO_L$, and let $\OO_L^{\times, +}$ be the subgroup of totally positive units. Thus,
$$\ker\left(\res_\infty|_{\OO_L^\times/(\OO_L^\times)^2}\right) = \OO_L^{\times,+}/(\OO_L^\times)^2.$$
We refer the reader to \cite{dummit} for heuristics and conjectures about the dimension of the totally positive units (in particular, the conjecture on page 4). In the following theorem, we use the notation of \cite{dummit}.
\begin{thm} \label{thm-rhos}
Let $\rho$, $\rho^+$, and $\rho_\infty$ be defined by
$$\rho = \dim_{\FF_2} \Cl(L)/2\Cl(L),\ \rho^+ = \dim_{\FF_2} \Cl^+(L)/2\Cl^+(L), \text{ and } \rho_\infty = \dim_{\FF_2} \OO_L^{\times,+}/(\OO_L^\times)^2$$
Then,
\begin{enumerate} 
\item $\rho_\infty= p-\dim_{\FF_2} \res_\infty(\OO_L^\times/(\OO_L^\times)^2)=\dim_{\FF_2} \{\pm 1 \}^{p}/\res_\infty(\OO_L^\times/(\OO_L^\times)^2).$
\item We have
$$ 0 \to \{\pm 1 \}^{p}/\res_\infty(\OO_L^\times/(\OO_L^\times)^2) \to \Cl^+(L) \to \Cl(L)\to 0.$$
In particular, $\max\{\rho,\rho_\infty\}\leq \rho^+\leq \rho_\infty+\rho$, and $\rho^+=\rho_\infty+\rho$ if and only if the exact sequence splits.
\item {\rm (Armitage-Fr\"ohlich)} $\rho^+-\rho \leq (p-1)/2$.
\end{enumerate}
\end{thm}
\begin{proof}
For part (1), note that $\rho_\infty = \dim_{\FF_2} \OO_L^{\times,+}/(\OO_L^\times)^2$ is the dimension of $\ker\left(\res_\infty|_{\OO_L^\times/(\OO_L^\times)^2}\right)$, and the dimension of $\OO_L^\times/(\OO_L^\times)^2$ is $p$. Thus, the dimension of the image of $\res_\infty|_{\OO_L^\times/(\OO_L^\times)^2}$ is $p$ minus the dimension of the kernel. 

For part (2), see Section 2 of \cite{dummit}, and in particular Equation (2.9). Part (3) is shown in \cite{armitage}, where it is shown that $\rho^+-\rho\leq \lfloor r_1/2 \rfloor$, where $r_1$ is the number of real embeddings of $L$. Here $r_1=p$ is an odd prime, so the proof is concluded.
\end{proof}

From the statement of the previous theorem, we see that $\rho^+\geq \rho_\infty$. However, $\rho\geq \rho_\infty$ is not necessarily true. In the following result, a condition is given that implies $\rho\geq \rho_\infty$ (see also \cite{hughes}, Section 3).

\begin{thm}[{\cite[Corollaire 2c]{oriat}}]\label{thm-oriat} Let $L/\QQ$ be a finite abelian extension with Galois group of odd exponent $n$, and suppose that $-1$ is congruent to a power of $2$ modulo $n$. Then, in the notation of Theorem \ref{thm-rhos}, we have $\rho=\rho^+$. In particular, $\rho\geq \rho_\infty$.
\end{thm}

We obtain the following corollary.

\begin{cor}\label{cor-order2iseven}
Let $L/\QQ$ be a cyclic extension of odd prime degree $p$, and suppose that the order of $2$ in $(\ZZ/p\ZZ)^\times$ is even. Then, $\rho=\rho^+$. In particular, $\dim_{\FF_2} \Cl(L)[2] = \dim_{\FF_2} \Cl^+(L)[2]$.
\end{cor}
\begin{proof}
Suppose that $\Gal(L/\QQ)\cong \ZZ/p\ZZ$ for some prime $p>2$, such that the order of $2$ in $(\ZZ/p\ZZ)^\times$ is even (since $\Gal(L/\QQ)$ is cyclic of order $p$, this is equivalent to $-1$ being congruent to a power of $2$ modulo $p$). Hence, Theorem \ref{thm-oriat} applies, and $\rho=\rho^+$. 
\end{proof}

The odd primes below $100$ such that the order of $2$ is odd modulo $p$ are $7$, 
$23$, $31$, $47$, $71$, $73$, $79$, and $89$, so the corollary applies to all other primes not in this list (i.e., $3$, $5$, $11$, $13$, $17$, $19$, etc.). 

\begin{thm}[\cite{garbanati}]\label{thm-garbanati}
Let $K/\QQ$ be an (imaginary) abelian extension of the rationals of degree $n$, let $h_K$ be the class number of $K$, and let $C_K\subseteq \OO_K^\times$ the groups of circular units of $K$ (as defined in \cite{garbanati}, p. 376), and let $C_K^+ \subseteq \OO_K^{\times, +}$ be the subgroup of circular units that are totally positive. Let $K^+$ be the maximal real subfield of $K$, and let $h_K^+$ be its class number. Let $h_K^- = h_K/h_K^+$. Further, assume that each prime $p$ which ramifies in $K$ does not split in $K^+$. Then:
\begin{enumerate}
\item The index $[\OO_K^{\times,+}:(\OO_K^\times)^2] $ is a divisor of the index $[C_K^{+}:C_K^2]$.
\item If the discriminant of $K$ is plus or minus a power of a prime, then $h_K^-$ is odd if and only if $[C_L^{+}:C_L^2]=1$, where $C_L$ is the subgroup of circular units of $L=K^+$.
\item Suppose the discriminant of $K$ is plus or minus a power of a prime, $[K^+:\QQ]=n$ is a power of an odd prime $p$, and the order of $2\bmod p$ is even. Then, $h_K^-$ is odd if and only if $h_K^+$ is odd.
\end{enumerate}
\end{thm}
\begin{proof} The results are, respectively, Lemma 5, Theorem 3, and Theorem 4 of \cite{garbanati}.
\end{proof}

Next, we cite a results of Estes which extends work  of Davis (\cite[Corollary 2]{davis}), and Stevenhagen (\cite[Theorem 2.5]{stevenhagen}). See also \cite[Theorems 3.1 and 3.3]{kim}.

\begin{thm}[\cite{estes}]\label{thm-estes} Let $q$ and $p$  be primes such that $q=2p+1$. If $2$ is inert in $\QQ(\zeta_p)^+$, where $\zeta_p$ is a primitive $p$-th root of unity, then the class number of $\QQ(\zeta_q)$ is odd.
\end{thm}

The following result combines the results of Davis, Estes, Stevenhagen, and Garbanati, and gives a specific criterion to check that $\rho_\infty=0$ for the maximal real subfield of a cyclotomic field.

\begin{thm}\label{thm-kim} Let $q$ and $p>2$  be primes such that $q=2p+1$, and let  $L=\QQ(\zeta_q+\zeta_q^{-1})$, where $\zeta_q$ is a primitive $q$-th root of unity. Further, assume that  the prime $2$ is inert in the extension $\QQ(\zeta_p)^+/\QQ$. Then, $\rho_\infty = \dim_{\FF_2} \OO_L^{\times,+}/(\OO_L^\times)^2 = 0.$
\end{thm}
\begin{proof}
Let $K=\QQ(\zeta_q)$ and let $K^+ = L = \QQ(\zeta_q+\zeta_q^{-1})$. We note that the discriminant of $K$ is a power of a prime (namely $q$), and therefore the primes of $K$ or $L$ that are ramified (namely the primes above $q$), are totally ramified, so they do not split. Moreover, $p=(q-1)/2$ is prime and $[L:\QQ]=p$. Thus, the hypotheses of Theorem \ref{thm-garbanati} are satisfied for $K$ and $L$.
	
    If $2$ is inert in $\QQ(\zeta_p)^+/\QQ$, then Theorem \ref{thm-estes} shows that $h_K$ is odd, and therefore $h_K^-$ is odd as well, since $h_K^- = h_K/h_K^+$ by definition. Since the discriminant of $K$ is a power of $q$, Theorem \ref{thm-garbanati} part (2) shows that $[C_L^+:C_L^2]=1$ and therefore $[\OO_L^{\times,+}:(\OO_L^\times)^2]=1$ as well by part (1). We conclude that $\rho_\infty=0$.
\end{proof} 

There is in fact a conjecture of Davis and Taussky that says that $\rho_\infty=0$ in the case of $L=\QQ(\zeta_q+\zeta_q^{-1})$, where $p=(q-1)/2$ is a Sophie Germain prime. For more on the Davis--Taussky conjecture see \cite{davis}, \cite{edgar}, \cite{estes}, \cite{kim}, and \cite{stevenhagen}.

\begin{conj}[Davis--Taussky conjecture]\label{conj-taussky}
Let $q$ and $p$  be primes such that $q=2p+1$, and let  $L=\QQ(\zeta_q+\zeta_q^{-1})$, where $\zeta_q$ is a primitive $q$-th root of unity. Then, $C^+_L=C_L^2$. (Thus, it follows that $\rho_\infty = \dim_{\FF_2} \OO_L^{\times,+}/(\OO_L^\times)^2 = 0.$)
\end{conj}

In the next result we note that the Davis--Taussky conjecture is equivalent to the class number of $\QQ(\zeta_q)$ being odd. (We thank David Dummit for pointing out the following equivalence to us.)

\begin{thm}\label{thm-davistausskyclassnumberodd}
	Let $q$ and $p$  be primes such that $q=2p+1$, let  $L=\QQ(\zeta_q+\zeta_q^{-1})$, where $\zeta_q$ is a primitive $q$-th root of unity, and let $K=\QQ(\zeta_q)$. Then, $C^+_L=C_L^2$ if and only if the class number of $K$ is odd.
\end{thm}
\begin{proof}
	If $h^+_K$ is even, then $h^-_K$ is even (see, for instance, \cite{stevenhagen}, pp. 773-774, for a proof of this fact). Since $h_K=h_K^+h_K^-$, it follows that $h_K^-$ is odd if and only if $h_K$ is odd. Since the discriminant of $K$ is a power of $q$ (prime), Theorem \ref{thm-garbanati}, part (2), implies that $h^-_K$ is odd if and only if $C^+_L = (C_L)^2$. Hence, the Davis--Taussky conjecture holds if and only if $h_K^-$ is odd, if and only if $h_K$ is odd.
\end{proof}

We conclude this section with some remarks about how to compute an upper bound of $\rho_\infty$ in the cyclic case, working in coordinates over $\FF_2$. Let $L$ be a cyclic extension of $\QQ$ of degree $p>2$, and let $\Gal(L/\QQ)=\langle \sigma \rangle$. Then, $L$ is totally real (since $L$ is Galois over $\QQ$, then it is totally real or totally imaginary, so $[L:\QQ]=p>2$ is odd, and $p=2r_2$ is impossible). Let $u\neq \pm 1$ be a fixed (known) unit in $\OO_L^\times$ and let 
$${\rm res}_\infty(u) = (\varepsilon_1,\ldots,\varepsilon_p)$$
where $\tau_1,\ldots,\tau_p$ are the real embeddings of $L$ and $\varepsilon_i$ is the sign of $\tau_i(u)\in \RR$. We order our embeddings in the following manner. Let $g_u(x)$ be the minimal polynomial of $u$ over $\QQ$, and let $r_1,\ldots,r_p$ be the real roots of $g_u(x)$ ordered so that $r_1<r_2<\cdots <r_p$. Then, $\{r_i\}=\{\tau_j(u)\}$, and we choose $\tau_i$ so that $\tau_i(u)=r_i$ for all $1\leq i\leq p$. With this notation, ${\rm res}_\infty(u) = (-1,-1,-1,\ldots,1,1,1)$, i.e., it consists of a non-negative number of $-1$ signs followed by a non-negative number of $+1$ signs. Recall that $u$ is in the kernel of ${\rm res}_\infty$ if and only if $u$ is a totally positive unit.
	
    Attached to  the generator $\sigma\in \mathcal{G}=\Gal(L/\QQ)$, there is a permutation $\phi=\phi_\sigma\in S_p$, where $\phi$ is considered here as a permutation of $\{1,2,\ldots,n\}$, such that $\tau_i(\sigma(u))=r_{\phi(i)}$, and therefore
    $${\rm res}_\infty(\sigma(u)) = (\varepsilon_{\phi(1)}, \ldots, \varepsilon_{\phi(p)}).$$
    Since $\tau$ is an embedding (injective), if $\tau_i(\alpha)=r_i$ for some $\alpha\in L^\times$, then $\tau_i(\sigma(\alpha)) = r_{\phi(i)}$. It follows that $\tau_i(\sigma^n(u)) = r_{\phi^n(i)}$, and 
       $${\rm res}_\infty(\sigma^n(u)) = (\varepsilon_{\phi^n(1)}, \ldots, \varepsilon_{\phi^n(p)}),$$
       for all $n\geq 1$.
       
Now, in addition, suppose that $u$ is a unit of norm $1$, and let $\mathcal{G}\cdot u$ be the subgroup of units generated by the conjugates of $u$, i.e., 
$$\mathcal{G}\cdot u = \langle u, \sigma(u), \sigma^2(u),\ldots, \sigma^{p-1}(u)\rangle \subseteq \OO_L^\times$$ generate a subgroup of  $\OO_L^\times$. Note that the product $\prod_{n=0}^{p-1} \sigma^n(u)=1$, so  $$\mathcal{G}\cdot u = \langle u, \sigma(u), \sigma^2(u),\ldots, \sigma^{p-2}(u)\rangle.$$ Then, 
$${\rm res}_\infty\left(\mathcal{G}\cdot u\right) = \langle 
(\varepsilon_{\phi^n(1)}, \ldots, \varepsilon_{\phi^n(p)}) : 0\leq n\leq p-2\rangle\subseteq V_\infty,$$
where we have defined $V_\infty = \{ \pm 1\}^p$. If we fix an isomorphism $\psi:\{\pm 1\}\cong \FF_2$, and write $f_i = \psi(\varepsilon_i)$, then the map ${\rm res}_\infty: \mathcal{G}\cdot u \to V_\infty$ can be written in $\FF_2$-coordinates, and the corresponding $p\times (p-1)$ matrix over $\FF_2$ is given by
$$M_{\infty,u} = \left( f_{\phi^j(i)}\right)_{\substack{1\leq i\leq p\\ 0\leq j \leq p-2}} = \left(\begin{array}{cccc} 
f_1 & f_{\phi(1)} & \cdots & f_{\phi^{p-2}(1)} \\ 
f_2 & f_{\phi(2)} & \cdots & f_{\phi^{p-2}(2)} \\
\vdots & \vdots & \ddots & \vdots\\
f_p & f_{\phi(p)} & \cdots & f_{\phi^{p-2}(p)} \\
\end{array} \right). $$
\begin{lemma}\label{lemma-m_infty}
Let $u$ be a unit of norm $1$, and let $d_{\infty,u}$ be the dimension of the column space of $M_{\infty,u}$ or, equivalently, the dimension of ${\rm res}_\infty(\mathcal{G}\cdot u)$. Then, $\rho_\infty \leq (p-1) - d_{\infty,u}$. In particular, if $d_{\infty,u} = p-1$, then $\rho_\infty =0$.
\end{lemma} 
\begin{proof}
If $u$ is of norm $1$, then $-1\not\in \mathcal{G}\cdot u$, because the norm of $-1$ is $-1$, and the norm of every element in $\mathcal{G}\cdot u$ is $1$. In particular, $\langle \res_\infty(-1), \res_\infty(\mathcal{G}\cdot u)\rangle$ is a space of dimension $1+d_{\infty,u}$. Hence, the kernel of $\res_\infty$ is at most of dimension $p-(1+d_{\infty,u})$. It follows that $\rho_\infty\leq (p-1)-d_{\infty,u}$ as desired.
\end{proof}

\subsection{Totally positive units in cyclic extensions of prime degree}\label{sec-totposprimedegree}

Let $L$ be a cyclic extension of prime degree $p>2$, and let $\OO_L^\times$ be the unit group of $\OO_L$. Let $\OO_L^{\times, 1}$ be the units of norm $1$, so that $\OO_L^\times \cong \{ \pm 1 \}\times \OO_L^{\times,1}$. In this section we show the following result:

\begin{prop}\label{prop-p3p5totpos} Let $p\geq 3$ be a prime, let $L$ be a cyclic extension of degree $p$, and suppose that the polynomial $\phi_p(x)=(x^p-1)/(x-1)$ is irreducible over $\FF_2$. Then, either $\rho_\infty=0$ (i.e., $\OO_L^{\times,+} = (\OO_L^{\times})^2$), or $\rho_\infty = p-1$ in which case every unit in $\OO_L^{\times,1}$ is totally positive.
\end{prop}
\begin{proof}
If every unit in $\OO_L^{\times,1}$ is totally positive, then $\rho_\infty=p-1$ since we would have
$$\rho_\infty = \dim_{\FF_2} \OO_L^{\times,+}/(\OO_L^\times)^2 = \dim_{\FF_2} \OO_L^{\times,1}/(\OO_L^\times)^2= p-1.$$
Otherwise, there must be a unit $u\in \OO_L^{\times,1}$ that is not totally positive (in particular, $u$ is not in $\pm 1 \cdot  (\OO_L^{\times})^2$). Let $G=\Gal(L/\QQ)=\langle \sigma \rangle \cong \ZZ/p\ZZ$, and let $u_i = \sigma^{i}(u)$ for $i=0,\ldots,p-1$, be the conjugates of $u$. Let $\tau$ be a fixed embedding of $L$ in $\RR$, let $\tau(u_i)=\varepsilon_i \in \{\pm 1 \}$ for $i=0,\ldots,p-1$, and order the embeddings $\tau=\tau_0,\ldots,\tau_{p-1}$ of $L$ such that $\res_\infty(u)=(\varepsilon_0,\ldots,\varepsilon_{p-1})$. In other words, $\tau_i = \tau\circ \sigma^i$. Thus, 
$$\res_\infty(\sigma(u)) = \res_\infty(u_1) = (\varepsilon_1,\ldots,\varepsilon_{p-1},\varepsilon_0).$$

Consider the class $\overline{u} \in \OO_L^{\times}/(\OO_L^\times)^2$ and its non-trivial signature $\res_\infty(u)$. The group ring $\FF_2[G]$ acts on the module $M=\FF_2[G]\cdot \overline{u}$. Since $G$ is of prime order and by assumption $\phi_p(x)$ is irreducible, it follows that $M$ is irreducible. Furthermore, since $u\neq \pm 1$ it follows that  $\res_\infty(u)\neq (1,1,\ldots,1)$ or $(-1,-1,\ldots,-1)$. Hence, $\res_\infty(\sigma(u))\neq \res_\infty(u)$ by our formula above, and therefore $M$ is not $1$-dimensional.

Moreover, 
$$\FF_2[G]\cong \FF_2[x]/(x^p-1)\cong \FF_2[x]/(x-1) \oplus \FF_2[x]/(\phi_p(x)),$$
Since we are assuming that $\phi_p(x)$ is irreducible over $\FF_2$, the only irreducible representations of $G$ over $\FF_2$ are the trivial ($1$-dimensional) representation, and a representation of dimension $p-1$. Since the irreducible $\FF_2[G]$-module $M$ is not $1$-dimensional, it must be $(p-1)$-dimensional. Finally, we note that $M \subseteq \OO_L^{\times,1}/(\OO_L^\times)^2$, and every unit class in $M$ has non-trivial signature (the norm is $1$ and there are $p>2$ signs, so both $1$ and $-1$ appear in the signature). Since the dimension of all possible signatures in $\OO_L^{\times,1}$ is $p-1$, and $M$ is $(p-1)$-dimensional, we conclude that all signatures occur, and therefore $\rho_\infty=0$, as desired. 
\end{proof}

We conclude this section quoting a conjectural density of cubic fields and quintic fields with maximal $\rho_\infty$, which is part of a broader conjecture of Dummit and Voight (see \cite{dummit}). 

\begin{conj}[{\cite{dummit}}]\label{conj-dummit}
Let $p=3$ (resp. $p=5$). As $L$ varies over all totally real fields of degree $p$ ordered by absolute discriminant, the density of such fields with $\rho_\infty=2$ (resp. $\rho_\infty=4$) is approximately $1.9\%$ (resp. $0.000019\%$). 
\end{conj}

\begin{remark}\label{rem-abundant} By Prop. \ref{prop-p3p5totpos}, if $L$ is a cyclic field of degree $p=5$, and $\rho_\infty\neq 4$, then it must be $0$. By Conjecture \ref{conj-dummit}, approximately $99.999981\%$ of all totally real quintic fields conjecturally have $\rho_\infty\neq 4$. Thus, we expect  that cyclic quintic fields with $\rho_\infty=0$ must be quite abundant. Note, however, that cyclic quintic fields are a subset of density $0$ among all totally real quintic fields, so the conjectures of Dummit and Voight do not apply directly here.
\end{remark}

\subsection{Refinements of the bound on the rank}\label{sec-muchbetterbounds}

Now we are ready to improve the bound in Proposition \ref{prop-Sel2_UB}. We will continue using the notation of Section \ref{sec-totally}.

\begin{prop}\label{prop1} Let $p$ be an odd prime, let $C: y^2=f(x)$ with $f(x)$ of degree $p$ (and genus $g=(p-1)/2$), such that $L$, the number field defined by $f(x)$, is totally real of degree $p$, and let $J/\QQ$ be the jacobian of $C/\QQ$. Let $\rho_\infty = \dim_{\FF_2} \OO_L^{\times,+}/(\OO_L^\times)^2$, and  let $j_\infty = \dim_{\FF_2} (\res_\infty(\OO_L^\times/(\OO_L^\times)^2)\cap J_\infty)$. Then:
\begin{align*}
\dim\Sel^{(2)}(\QQ,J) \leq j_\infty +\rho_\infty + \dim \Cl^+(L)[2] + \dim G\cap \val\left(L_{J_\infty}\right),
\end{align*}
In particular, 
\begin{enumerate}
\item $\rho_\infty+j_\infty\leq p-1$.
\item $j_\infty \leq \dim J_\infty = (p-1)/2 = \operatorname{genus}(C)$.
\item $\dim \Cl^+(L)[2] \leq \rho_\infty + \dim \Cl(L)[2]$. In particular,
\begin{align*}
\dim\Sel^{(2)}(\QQ,J) \leq j_\infty +  2\rho_\infty + \dim \Cl(L)[2] + \dim G\cap \val\left(L_{J_\infty}\right),
\end{align*}
\item If $G\cap \val\left(L_{J_\infty}\right)$ is trivial, then $$\dim\Sel^{(2)}(\QQ,J) \leq j_\infty + \rho_\infty +  \dim \Cl^+(L)[2]\leq j_\infty +2\rho_\infty +  \dim \Cl(L)[2].$$
\end{enumerate}
\end{prop}
\begin{proof}
Let $p$ be an odd prime, let $C: y^2=f(x)$ with $f(x)$ of degree $p$, such that $L$, the number field defined by $f(x)$, is totally real of degree $p$, and let $J/\QQ$ be the jacobian of $C/\QQ$. Recall that in Section \ref{sec-betterbounds} we showed
\begin{align*}
\dim\Sel^{(2)}(\QQ,J)& \leq \dim(\widehat{H}) = \dim U + \dim V + \dim W \\ 
& = \dim \left(\OO_L^\times/(\OO_L^\times)^2\cap \res_\infty^{-1}(J_\infty)\right) + \dim \Cl(L_{J_\infty})[2] + \dim G\cap \text{val}\left(L_{J_\infty}\right).
\end{align*}
Clearly,
\begin{align*} \dim \left(\OO_L^\times/(\OO_L^\times)^2\cap \res_\infty^{-1}(J_\infty)\right) &= \dim \res_\infty(\OO_L^\times/(\OO_L^\times)^2)\cap J_\infty + \dim \ker (\res_\infty|_{\OO_L^\times/(\OO_L^\times)^2})\\
 &= j_\infty + \dim \OO_L^{\times,+}/(\OO_L^\times)^2 = j_\infty + \rho_\infty. 
\end{align*}
Now, for part (1), notice that $\left(\OO_L^\times/(\OO_L^\times)^2\cap \res_\infty^{-1}(J_\infty)\right)\subseteq \OO_L^\times/(\OO_L^\times)^2$ so the dimension as a $\FF_2$-vector space is at most $p$. Moreover, $\res_\infty(-1)$ is not in $J_\infty$ (because $J_\infty=\delta_\infty(J(\RR))\subseteq H_\infty$, which is the kernel of the norm map, so $N_\QQ^L(j)=1$ for $j\in L$ such that $\res_\infty(j)\in J_\infty$, but $N_\QQ^L(-1)=-1$ since the degree of $L$ is odd). Thus, $j_\infty+\rho_\infty = \dim \left(\OO_L^\times/(\OO_L^\times)^2\cap \res_\infty^{-1}(J_\infty)\right)\leq p-1$, as claimed.

For part (2), recall that we have defined $J_\infty = \delta_\infty(J(\RR)/2J(\RR))\subseteq H_\infty$. By Lemma \ref{lem-Real_dims}, and the fact that $\delta_\infty$ is injective (Lemma 4.1 in \cite{stoll}), we have 
$$\dim(J_\infty) = m_\infty-1-g = p-1 - \frac{p-1}{2} = \frac{p-1}{2},$$
where we have used the fact that $L$ is totally real to claim that $m_\infty=p$.

Part (3) follows from Theorem \ref{thm-rhos}, which shows that $\rho^+\leq \rho_\infty + \rho$. And part (4) is immediate from (3), so the proof is complete.
\end{proof} 

Now we can put together Corollary \ref{cor-order2iseven} and Proposition \ref{prop1} to give a bound in the cases when the multiplicative order of $2\bmod p$ is even.

\begin{thm}\label{thm-totally_pos} Suppose $L$ is a cyclic, totally real number field of degree $p>2$, such that the order of $2$ in $(\ZZ/p\ZZ)^\times$ is even. Then, in the notation of Proposition \ref{prop1}, we have
\begin{align*}
\dim\Sel^{(2)}(\QQ,J) & \leq j_\infty + \rho_\infty + \dim \Cl(L)[2] + \dim\ker\big(G\to\Cl(L)/2\Cl(L)\big)\\
& \leq p-1 + \dim \Cl(L)[2] + \dim\ker\big(G\to\Cl(L)/2\Cl(L)\big).
\end{align*}
Moreover, if $\rho_\infty=0$, then 
\begin{align*}
\dim\Sel^{(2)}(\QQ,J) & \leq g + \dim \Cl(L)[2] + \dim\ker\big(G\to\Cl(L)/2\Cl(L)\big)\\
& = \frac{p-1}{2} + \dim \Cl(L)[2] + \dim\ker\big(G\to\Cl(L)/2\Cl(L)\big).
\end{align*}
\end{thm}
\begin{proof}
Suppose $L$ is a cyclic, totally real number field of degree $p>2$, such that the order of $2$ in $(\ZZ/p\ZZ)^\times$ is even. Then, Corollary \ref{cor-order2iseven} implies that $\rho=\rho^+$. Thus, the bound follows from Proposition \ref{prop1}. Note that the bound $p-1+ \dim \Cl(L)[2] + \dim\ker\big(G\to\Cl(L)/2\Cl(L)\big)$ is the bound that appears in Proposition \ref{stollbnd}.
\end{proof}

\section{Genus 1}\label{sec-genus1}

The goal of this section is to show an alternative proof of Theorem \ref{thm-wash}, using Stoll's implementation of the $2$-descent algorithm and the results we showed in the previous section. Let $m$ be an integer such that $D=m^2+3m+9$ is square-free. Let $C$ be the (hyper)elliptic curve given by the Weierstrass equation
\[
C: y^2 = f_m(x)= x^3+mx^2-(m+3)x+1.
\]
Since $C$ is elliptic, the jacobian $J$ is isomorphic to $C$, so we will identify $C$ with $J$. We will conclude the bound stated in the theorem as a consequence of Theorem \ref{thm-totally_pos}. Since the order of $2\equiv -1 \bmod 3$ is $2$, Corollary \ref{cor-order2iseven} shows that $\rho^+=\rho$. In Section \ref{sec-totposprimedegree} we discussed that cyclic cubic fields with $\rho_\infty\neq 0$ are rare. Let us show that in fact $\OO_L^{\times,+} = (\OO_L^\times)^2$, and therefore $\rho_\infty=0$, for the cyclic cubic fields $L=L_m$ defined by $f_m(x)=0$.

\begin{lemma}[{\cite[\S 1, p. 371]{Wash1}}]\label{lem-wash} Let $m$ be an integer such that $m\equiv 3 \bmod 9$, and let $L$ be the number field defined by $f_m(x)= x^3+mx^2-(m+3)x+1=0$. Then, $\rho_\infty=0$.
\end{lemma}
\begin{proof}
Let $\alpha$ be the negative root of $f_m(x)$. Then $\alpha'=1/(1-\alpha)$, and $\alpha''=1-1/\alpha$ are the two other roots, and in fact they are units in $\OO_L^\times$. Moreover,
$$-m-2<\alpha<-m-1<0<\alpha'<1<\alpha''<2$$
and therefore all eight possible sign signatures may be obtained from $\alpha$ and its conjugates. Thus, every totally positive unit is a square, and $\rho_\infty=0$, as claimed. 
\end{proof} 

Thus, in order to prove Theorem \ref{thm-wash}, it is enough to show that $G$ is trivial.

\begin{lemma} \label{lem-genus1b}
Let $m\geq 0$ be an integer such that $D=m^2+3m+9$ is square-free. Let $v=2$, or let $v$ be a prime divisor of $D$, and let $f_m(x)=x^3+mx^2-(m+3)x+1$. Then, $f_m(x)$ is irreducible as a polynomial in  $\QQ_v[x]$.
\end{lemma}
\begin{proof}
Let $v=2$. Then, if we consider $f_m(x)$ as a polynomial in $\FF_2[x]$, we have
\[
f_m(x)=\begin{cases}
x^3+x^2+1&\text{if }m\equiv 1 \mod 2,\\
x^3+x+1&\text{if }m\equiv 0 \mod 2.
\end{cases}
\]
In both cases, $f_m$ is irreducible over $\FF_2$, hence it is irreducible over $\QQ_2$.

Now, let $v>2$ be a prime divisor of $D$. Then, by assumption, $v>3$ and
\[
f_m(x-m/3)=x^3-\frac{D}{3}x+\frac{D\cdot (2m+3)}{27}
\]
is integral over $\ZZ_v$. Since $D=m^2+3m+9$, it follows that $4D-(2m+3)^2=27$ and therefore the greatest common divisor of $D$ and $2m+3$ divides $27$. Since by assumption $D$ is square-free, it cannot be divisible by 3, and so $\gcd(D,2m+3)=1$; this together with the fact that $D$ is square-free implies that $f_m(x-m/3)$ is Eisenstein over $\ZZ_v$. Hence, $f_m$ is irreducible over $\QQ_v$, as claimed.
\end{proof}

We are now ready to prove Theorem \ref{thm-wash}.

\begin{theorem}[{\cite[Theorem 1]{Wash1}}]\label{thm-wash-2}
Let $m\geq 0$ be an integer such that  $m^2+3m+9$ is square-free. Let $E_m$ be the elliptic curve given by the Weierstrass equation
$$E_m: y^2 = f_m(x)= x^3+mx^2-(m+3)x+1.$$
Let $L_m$ be the number field generated by a root of $f_m(x)$, and let $\operatorname{Cl}(L_m)$ be its class group. Then,
$$\operatorname{rank}_\ZZ(E_m(\QQ)) \leq 1 + \operatorname{dim}_{\FF_2}(\operatorname{Cl}(L_m)[2]).$$
\end{theorem}

\begin{proof}
We shall use Theorem \ref{thm-totally_pos}. The order of $2\equiv -1 \bmod 3$ is $2$, even, so $\rho=\rho^+$, and Lemma \ref{lem-wash} shows that $\rho_\infty=0$, so it remains to compute $G_m = \prod_{v\in S_m\setminus \{\infty\}} G_{m,v}$. 

The discriminant of $f_m$ is $D^2=(m^2+3m+9)^2$, so we have
\[
S_m=\{\infty,2\}\cup\{v\mid D\}.
\]
However, by Lemma \ref{lem-genus1b}, the polynomial $f_m(x)$ is irreducible over $\QQ_v$ for any finite prime $v\in S_m$. It follows that the number of irreducible factors of $f_m(x)$ over $\QQ_v$ is $1$, and  therefore $I_{m,v}$ is zero-dimensional by Lemma \ref{lem-local_dims}. Since $G_{m,v}\subseteq I_{m,v}$, we conclude that $G_{m,v}$ is always trivial. Hence, $G_m$ is trivial, and Theorem \ref{thm-totally_pos} implies that 
\begin{align*}
\dim\Sel^{(2)}(\QQ,J_m) \leq g + \dim \Cl(L_m)[2] + \dim\ker\big(G_m\to\Cl(L_m)/2\Cl(L_m)\big) = g + \dim \Cl(L_m)[2],
\end{align*}
where $J_m$ is the jacobian of the elliptic curve $E_m$. Since the genus of $E_m$ is $1$, then $E_m \cong J_m$ over $\QQ$. Hence,
\begin{align*}
\rank_\ZZ(E_m(\QQ))\leq \dim\Sel^{(2)}(\QQ,E_m) \leq 1 + \dim \Cl(L_m)[2],
\end{align*}
as desired.
\end{proof}

\section{Genus $g=(p-1)/2$, where $p$ is a Sophie Germain prime}\label{sec-sophie}

The goal of this section is to find examples of hyperelliptic curves of genus $g\geq 2$ where the dimension of the Selmer group can be bounded in terms of a class group, as in Theorem \ref{thm-wash} (the genus $1$ case). We begin by looking at polynomials $f(x)$ that cut out extensions of degree $p$, contained inside a $q$-th cyclotomic extension, where $q$ is another prime.

\begin{thm}\label{thm-powerbasis}
Let $q>2$ be a prime such that the multiplicative order of $2\bmod q$ is either $q-1$ or $(q-1)/2$, and let $p>2$ be a prime dividing $q-1$. Let $\QQ(\zeta_q)$ be the $q$-th cyclotomic field, and let $L$ be the unique extension of degree $p$ contained in $\QQ(\zeta_q)$. Further, suppose that $\OO_L = \ZZ[\alpha]$ for some algebraic integer $\alpha\in L$, let $f(x)$ be the minimal polynomial of $\alpha$, and let $J/\QQ$ be the jacobian variety associated to the hyperelliptic curve $C: y^2=f(x)$. Then,
\begin{align*}
\dim\Sel^{(2)}(\QQ,J) \leq \rho_\infty + g + \dim \Cl^+(L)[2].
\end{align*}
If in addition the multiplicative order of $2\bmod p$ is even, then $$\dim\Sel^{(2)}(\QQ,J) \leq \rho_\infty + g + \dim \Cl(L)[2].$$
\end{thm}
\begin{proof}
We apply Stoll's algorithm to $y^2=f(x)$, in order to compute $G$. Note that $f(x)$ is monic, integral, and irreducible over $\QQ$, since $\alpha$ generates $\OO_L$ as a ring. Moreover, since $\OO_L=\ZZ[\alpha]$, it also follows that
$$\disc(f(x))= \disc(\mathcal{O}_L),$$
and since $L\subseteq \QQ(\zeta_q)$, the only prime dividing $\disc(\OO_L)$ is $q$. Hence, the set $S=\{\infty,2,q\}$. We will show that $G_2$ and $G_q$ are trivial, and therefore $G$ is trivial as well. Indeed:
\begin{itemize}
\item Let $v=2$. Since the order of $2$ modulo $q$ is $q-1$ or $(q-1)/2$ by hypothesis, it follows from \cite[Theorem 26]{marcus} that $2$ splits into $1$ or $2$ prime ideals in $\ZZ(\zeta_q)/\ZZ$, and therefore $2$ must be inert in the intermediary extension $L/\QQ$ of degree $p$. In particular, the polynomial $f(x)$ is irreducible over $\QQ_2$, since it defines an unramified extension $L_2/\QQ_2$ of degree $p$. Hence, $m_2=1$, and the dimension of $I_2$ is $m_2-1=0$ by Lemma \ref{lem-local_dims}. Since $G_2\subseteq I_2$, we conclude that $G_2$ is trivial as well.

\item Let $v=q$. Since $L\subseteq \QQ(\zeta_q)$ and $q$ is totally ramified in the cyclotomic extension, it is also totally ramified in $L/\QQ$. Thus, $f(x)$ is irreducible over $\QQ_q$ because it defines a totally ramified extension $L_q/\QQ_q$ of degree $p$. Thus, $m_q=1$ and arguing as above in the case of $v=2$, we conclude that $G_q$ is trivial.
\end{itemize}
Since the only finite primes in $S$ are $2$ and $q$, it follows that $G=G_2\times G_q$ is trivial. Now, Proposition \ref{prop1} shows the bound $\dim\Sel^{(2)}(\QQ,J) \leq \rho_\infty + g + \dim \Cl^+(L)[2].$ If in addition the order of $2\bmod p$ is even, and since $L/\QQ$ is cyclic of degree $p$, then Corollary \ref{cor-order2iseven} shows that $\rho=\rho^+$. Hence $\dim\Sel^{(2)}(\QQ,J) \leq  \rho_\infty + g + \dim \Cl(L)[2],$ as claimed.
\end{proof}

The drawback, however, of the previous result is that there are very few subfields of cyclotomic extensions with a power basis, as the following result points out.

\begin{thm}[Gras, \cite{gras}]\label{thm-gras}
Let $L$ be an extension of degree $p\geq 5$, and let $\OO_L$ be the maximal order of $L$. Then, $\OO_L$ has a power basis if and only if $L=\QQ(\zeta_q)^+$ is the maximal real subfield of the $q$-th cyclotomic field, where $q$ is a prime with $q=2p+1$. 
\end{thm}

For instance, the unique cyclic number field of degree $5$ with a power basis for the maximal order is $\QQ(\zeta_{11})^+$. Also, there is no cyclic number field of degree $7$ with a power basis for its maximal order (since $15$ is not a prime). Hence, we concentrate on those cyclic extensions of degree $p$, where $p$ is a Sophie Germain prime, i.e., $q=2p+1$ is also prime. 

\begin{thm}\label{thm-sophie}
Let $q\geq 7$ be a prime such that $p=(q-1)/2$ is also prime, and let $L=\QQ(\zeta_q)^+$ be the maximal real subfield of $\QQ(\zeta_q)$. Let $f(x)\in \ZZ[x]$ be any monic integral polynomial defining $L$, let $C: y^2=f(x)$, and let $J/\QQ$ be its jacobian. Then,  
\begin{align*}
\dim\Sel^{(2)}(\QQ,J) \leq \rho_{\infty} + j_\infty + \dim \Cl^+(L)[2] + \dim\ker\big(G\to\Cl(L)/2\Cl(L)\big).
\end{align*}
Moreover, if $f(x)$ is the minimal polynomial of $\zeta_q+\zeta_q^{-1}$ or $-(\zeta_q+\zeta_q^{-1})$, then 
\begin{align*}
\dim\Sel^{(2)}(\QQ,J) \leq \rho_{\infty} + j_\infty + \dim \Cl^+(L)[2].
\end{align*}
Further, if one of the following conditions is satisfied, 
\begin{enumerate}
\item the Davis--Taussky conjecture holds, or
\item the prime $2$ is inert in the extension $\QQ(\zeta_p)^+/\QQ$,
\item  $q\leq 92459$,
\end{enumerate}
 then $\rho_{\infty}=0$, and $\dim\Sel^{(2)}(\QQ,J) \leq  g + \dim \Cl(L)[2].$
\end{thm}
\begin{proof} The first bound follows from Proposition \ref{prop1}, so let us assume that $f(x)$ is the minimal polynomial of $\zeta_q+\zeta_q^{-1}$ or $-(\zeta_q+\zeta_q^{-1})$.  The ring of integers $\OO_{\QQ(\zeta_q)^+}$ has a power basis, namely $\ZZ[\zeta_q+\zeta_q^{-1}]$. Moreover, since $p$ is a Sophie Germain prime (with $q=2p+1$ prime), it follows that the multiplicative order of $2\bmod q$ is a divisor of $2p = 2\cdot ((q-1)/2)$. Since $q\geq 7$, the order of $2$ is bigger than $2$, so it must be $p=(q-1)/2$ or $q-1$. Thus, Theorem \ref{thm-powerbasis} applies and we obtain $\dim\Sel^{(2)}(\QQ,J) \leq \rho_{\infty} + j_\infty + \dim \Cl^+(L)[2].$

Further, if (1), (2), or (3) holds, then by Conjecture \ref{conj-taussky}, or Theorem \ref{thm-kim}, respectively, we find that $\rho_\infty=0$ and $\rho=\rho^+$.

If $q\leq 92459$, we shall use the computational approach at the end of Section \ref{sec-totally} to show that $\rho_\infty=0$. Let $\zeta=\zeta_q=e^{2\pi i/q}$. Then, the ring of integers of $\QQ(\zeta_q)^+$ has a power basis, namely $\OO_{\QQ(\zeta_q)^+} = \ZZ[\zeta_q+\zeta_q^{-1}]$. Let $u=-(\zeta+\zeta^{-1})$ if $p\equiv 1 \bmod 4$ and $u=\zeta+\zeta^{-1}$ if $p\equiv 3 \bmod 4$, thus chosen so that $u$ is a unit in $\OO_L^\times$ of norm $1$. Moreover, we note that 
if $p\equiv 1 \bmod 4$, then
$$-(\zeta+\zeta^{-1}) < -(\zeta^{2}+\zeta^{-2}) < \cdots < -(\zeta^{\frac{p-1}{2}}+\zeta^{-\frac{p-1}{2}}) < 0 < -(\zeta^{\frac{p+1}{2}}+\zeta^{-\frac{p+1}{2}}) < \cdots < -(\zeta^p+\zeta^{-p}) <1,$$
and if $p\equiv 3 \bmod 4$, then
$$\zeta^p+\zeta^{-p} < \zeta^{p-1}+\zeta^{-(p-1)} < \cdots < \zeta^{\frac{p+1}{2}}+\zeta^{-\frac{p+1}{2}} < 0 < \zeta^{\frac{p-1}{2}}+\zeta^{-\frac{p-1}{2}} < \cdots < \zeta+\zeta^{-1} <1.$$
Thus, according to our conventions described in this section, the embeddings $\tau_1,\ldots, \tau_p$ are numbered so that $\tau_i(u)=r_i\in \RR$ with 
$$r_i = \begin{cases}
-(\zeta^i+\zeta^{-i}) & \text{ if } p\equiv 1 \bmod 4,\\
\zeta^{p+1-i}+\zeta^{-(p+1-i)} & \text{ if } p\equiv 3\bmod 4,
\end{cases}$$
for all $1\leq i \leq p$. Thus,
$${\rm res}_\infty(u) = 
 (-1,-1,-1,\ldots,1,1,1) \in H_\infty $$
 with $(p-1)/2$ minus ones when $p\equiv 1 \bmod 4$, and $(p+1)/2$ minus ones when $p\equiv 3 \bmod 4$.
 
 Now, the Galois group $\mathcal{G}=\Gal(L/\QQ)$ is cyclic of order $p$. Since $q=2p+1$ is prime, then the multiplicative order of $2\bmod q$ is $(q-1)/2$ or $q-1$. Thus, either $-2$ or $2$ is a primitive root mod $q$. Let $\gamma \colon \QQ(\zeta)\to \QQ(\zeta)$ that sends $\gamma(\zeta)=\zeta^2$. It follows that $\sigma = \overline{\gamma} \in \mathcal{G} \cong (\ZZ/q\ZZ)^\times/\{\pm 1\}$ is a generator. Thus, the automorphism $\sigma(\zeta+\zeta^{-1})=\zeta^2+\zeta^{-2}$ generates $\mathcal{G}$. Let $\phi=\phi_\sigma\in S_p$ be the permutation attached to $\sigma$ as defined above. For instance, if $p\equiv 1 \bmod 4$, then $r_1=-(\zeta+\zeta^{-1})$, where $\zeta=e^{2\pi i/q}$, so $\tau_1(\sigma(u))=r_2=-(\zeta^2+\zeta^{-2})$  and therefore $\phi(1)=2$. However, if $p\equiv 3\bmod 4$, then $r_1=\zeta^p+\zeta^{-p}$. We can find an integer $1\leq k\leq p$, and $k$ or $-k\equiv 2p \bmod q$, such that  $\sigma(u)=\zeta^k+\zeta^{-k}$. It follows that $\tau(\sigma(u))=r_k$ and so $\phi(1)=k$ in this case. In general, the permutation $\phi$ is defined by
 $$\phi(i) = \min\{2\cdot i \bmod q,\ (-2\cdot i) \bmod q\}$$
 when $p\equiv 1 \bmod 4$, and by
 $$\phi(i) = p+1 - \min\{(2\cdot (p+1-i)) \bmod q,\ q-(2\cdot (p+1-i) \bmod q)\},$$
 when $p\equiv 3 \bmod 4$, where our representatives in $\ZZ/q\ZZ$ are always chosen amongst $\{0,1,\ldots,q-1\}$.
	With these explicit descriptions of ${\rm res}_\infty(u)$ and $\phi_\sigma$, we have computed  (using Magma) the matrix $M_{\infty,u}$ for all primes $p$ and $q$, with $q\leq 92459$, as in the statement, and in all cases $d_{\infty,u}=p-1$. Hence, $\rho_\infty=0$ follows from Lemma \ref{lemma-m_infty}.
\end{proof}

\begin{remark}\label{rem-davistausskyoddclassnumber}
	If the Davis--Taussky conjecture holds, then the class number $h_K$ of $\QQ(\zeta_q)$ is odd (by Theorem \ref{thm-davistausskyclassnumberodd}), and therefore the class number $h_K^+$ of $L=\QQ(\zeta_q)^+$ is odd as well (because $h_K^+$ is a divisor of $h_K$). Hence, if the Davis--Taussky conjecture holds, then $\Cl(L)[2]$ is trivial, and the bound of Theorem \ref{thm-sophie} becomes 
	$$\dim\Sel^{(2)}(\QQ,J) \leq  g.$$
\end{remark}

\section{Examples} \label{sec-examples}

\subsection{Curves of Genus $1$}\label{sec-examplesgenus1}

In this section we present some data that was collected on the curves described in Theorem \ref{thm-wash} (the data collected can be found at \cite{danielsweb}). For $m\in\ZZ$, let $f_m(x)$, $E_m,$ and $L_m$ be as in Theorem \ref{thm-wash}. Using Magma, we attempted to compute (subject to GRH) the Mordell--Weil rank of $E_m(\QQ)$ and $\dim_{\FF_2}\Cl(L_m)[2]$ for every $m\in \ZZ$ such that $1\leq m \leq 20000$ and $m^2+3m +9$ is square-free. There are $12462$ such values of $m$ in the given interval, and we were able to compute the rank of $E_m/\QQ$ for  $12235$ of them. For the other 227 curves, we were only able to get upper and lower bounds on their rank.

Of the $12462$ curves that we tested, $10327$ of them (about $82.87\%$) had rank equal to the upper bound given in Theorem \ref{thm-wash}. However, one might expect that the sharpness of this upper bound would decay as $m$ gets larger and larger, and in fact that seems to be the case. Let us define a function to keep track of the sharpness of the bound in an interval. 
\begin{defn}
Let $$M = \{m\in \ZZ : 1\leq m \leq 20000 \hbox{ and }m^2+3m +9 \hbox{ is square-free}\} $$ 
and
$S = \{m\in M : \rank(E_m) = \dim \Cl(L_m) +1\}$. Given $I\subseteq M$, we define  $\displaystyle \Sharp(I) = \frac{\#(S\cap I)}{\#(M\cap I)}.$
\end{defn}
\begin{remark}
The set $S$ only includes curves whose rank was actually computable and, because of this, the $\Sharp(I)$ statistic only gives a lower bound for how sharp Washington's upper bound is over the set $I$. 
\end{remark}
 
In order to see how the sharpness of the upper bound degrades as $m$ grows, in Table \ref{tab-sharpness} we present $\Sharp(I)$ over disjoint intervals of length $1000$. In the data, we clearly see that the number of curves for which Washington's bound is sharp in a given interval does start to decrease, but the bound is still sharp more often than not (notice, however, that the sharpness is inflated by the fact that the bound is sharp every time the bound is $1$, since there is a point of infinite order $P = (0,1)$ on $E_m$). 

\begin{table}
\begin{center}
\renewcommand{\arraystretch}{1.2}
\begin{tabular}{|c|l||c|l|}\hline
$I$ & $\Sharp(I)$&$I$ & $\Sharp(I)$\\\hline
$[ 1 , 1000 ]$& 0.91451    & $[ 10001 , 11000 ]$& 0.81862\\
$[ 1001 , 2000 ]$& 0.88499 & $[ 11001 , 12000 ]$& 0.81451\\
$[ 2001 , 3000 ]$& 0.87581 & $[ 12001 , 13000 ]$& 0.79936\\
$[ 3001 , 4000 ]$& 0.84665 & $[ 13001 , 14000 ]$& 0.81833\\
$[ 4001 , 5000 ]$& 0.82051 & $[ 14001 , 15000 ]$& 0.81760\\
$[ 5001 , 6000 ]$& 0.82504 & $[ 15001 , 16000 ]$& 0.78560\\
$[ 6001 , 7000 ]$& 0.84911 & $[ 16001 , 17000 ]$& 0.82664\\
$[ 7001 , 8000 ]$& 0.84455 & $[ 17001 , 18000 ]$& 0.82258\\
$[ 8001 , 9000 ]$& 0.81862 & $[ 18001 , 19000 ]$& 0.80512\\
$[ 9001 , 10000 ]$& 0.80000 & $[ 19001 , 20000 ]$& 0.78583\\\hline
\end{tabular}
\end{center}
\caption{Measure of the sharpness of the bound presented in Theorem \ref{thm-wash}.}  \label{tab-sharpness}
\end{table}
\renewcommand{\arraystretch}{1}

To see how fixing the bound first affects its sharpness, we define the following sets
$$T(r) = \{m \in M:\rank(E_m) =  r\}\hbox{ and } B(b) = \{m \in M: \dim\Cl(L_m)[2]+1 = b\}.$$ In Table \ref{tab-given_bound}, for each bound $b$ that occurs we give the number of curves of rank $r$ whose bound is $b$ for each $r$ that occurs. We also give the percentage of the curves whose rank is exactly $b$ and provide the totals of each column so that we can see how many curves of each rank we found (for similar statistics and conjectures in a broader context, see \cite{lozano}).

\begin{table}
\begin{center}
\renewcommand{\arraystretch}{1.3}
\begin{tabular}{|c||c|c|c|c||c|}\hline
$b$& $\#(T(1)\cap B(b))$ & $\#(T(3)\cap B(b))$&$\#(T(5)\cap B(b))$&$\#(T(7)\cap B(b))$ & $\frac{\#(T(b)\cap B(b))}{\#B(b)}$ \\\hline
1 & 7391 & 0 & 0 & 0  &  1.0000  \\\hline
3 & 1565 & 2809 & 0 & 0  & 0.6422   \\\hline
5 & 37 & 298 & 125 & 0  & 0.2717   \\\hline
7 & 0 & 1 & 7 & 2   & 0.2000  \\\hline\hline
Totals & 8993 & 3108 & 132 & 2   & ------\\\hline

\end{tabular}
\end{center}
\caption{Sharpness for a fixed rank bound $b$ in the interval $1\leq m \leq 20000$.}  \label{tab-given_bound}
\end{table}
\renewcommand{\arraystretch}{1}

From Table \ref{tab-given_bound} we can see that all of the curves that we computed, have odd rank less than or equal to 7. It also turns out that for all of the curves that we computed, Washington's bound is also odd and less than or equal to 7. In Table \ref{tab-first_m}, we give the first $m$ such that $\rank(E_m) = r$ and $\dim \Cl(L_m)+1= b$ for each pair $(r,b)$ that occurred. 

It is also interesting to point out that the average rank among curves with $b=1$ is $1$, the average rank among curves with $b=3$ is $2.23$, among curves with $b=5$ is $3.38$, and among curves with $b=7$ the average rank is $5.20$ (see \cite{lozano} for other examples of {\it Selmer bias} in genus $1$). 

\begin{table}
\begin{center}
\renewcommand{\arraystretch}{1.2}
\begin{tabular}{|c|l||c|l||c|l|}\hline
$m$ & $(r,b)$& $m$ & $(r,b)$& $m$ & $(r,b)$ \\\hline 
$ 1 $ & $( 1, 1 )$&$ 11$ & $(3, 3 )$&$ 143$ & $(5, 5 )$\\
$ 170$ & $(1, 3)$&$ 157$ & $(3, 5 )$&$ 3461$ & $(5, 7 )$\\
$ 2330$ & $(1, 5 )$&$ 19466$ & $(3, 7)$&$ 12563$ & $(7, 7 )$\\\hline
\end{tabular}
\end{center}
\caption{The first occurrence of each (rank, bound) pair that occurs for some $m\leq 20000$.}  \label{tab-first_m}
\end{table}
\renewcommand{\arraystretch}{1}
\begin{example}
Lastly, for the sake of concreteness, we end this section with an explicit example. When $m=143$ we have that 
$$E_{143}:y^2 = x^3 + 143x^2 - 146x + 1 $$
with conductor $2^2\cdot 20887^2$. Using Magma, we can compute that 
$$\Cl(L_{143})\cong \ZZ/2\ZZ + \ZZ/2\ZZ + \ZZ/4\ZZ + \ZZ/4\ZZ,$$
and so Washington's bound for the Mordell--Weil rank is $\dim\Cl(L_{143})[2] +1 = 4+1 = 5.$ Looking for points on $E_{143}$ we find $5$ independent points of infinite order that generate the Mordel--Weil group. 
$$E_{143}(\QQ) \cong \ZZ^5 = \langle (126/121 , -3023/1331), (90 , -1369), (65/64 , 577/512), (21/4 , 
-461/8), (-1 , 17) \rangle.$$
\end{example}

\subsection{Examples in the Sophie Germain case}\label{sec-examplessophie}

In this section we show examples of hyperelliptic curves that arise from Theorem \ref{thm-sophie}. 

\begin{example}
Let $q=7$ and $p=3$. Then, $L=\QQ(\zeta_7)^+$ is the maximal real subfield of $\QQ(\zeta_7)$, which has degree $3$, and class number $1$ (see \cite[Tables, \S 3]{Wash2}). Note that the order of $2$ in $(\ZZ/3\ZZ)^\times$ is $2=p-1$, and therefore condition (2) of Theorem \ref{thm-sophie} is met. Hence, if $f(x)$ is the minimal polynomial of $\pm(\zeta_7+\zeta_7^{-1})$, then the jacobian $J(\QQ)$ of $y^2=f(x)$ satisfies
$$\rank_\ZZ(J(\QQ))\leq \dim_{\FF_2} \Sel^{(2)}(\QQ,J) \leq g+\dim \Cl(L)[2] = 1.$$
For instance, if $f(x)$ is chosen to be the minimal polynomial of $-(\zeta_7+\zeta_7^{-1})$, then $f(x)=x^3 - x^2 - 2x + 1$ we in fact recover the elliptic curve $E_{-1}$ of Theorem \ref{thm-wash} for $m=-1$. The Mordell--Weil rank of the elliptic curve $E_{-1}\colon y^2=x^3 - x^2 - 2x + 1$ is $1$, so the bound on the rank given by Theorem \ref{thm-sophie} in this case is in fact sharp.
\end{example}

\begin{example} \label{ex-pequal5b}
If $q=11$ and $p=5$, then $L=\QQ(\zeta_{11})^+$ is a field of degree $5$ and trivial class group. Since $2$ is a primitive root modulo $5$, if $f(x)$ be the minimal polynomial of $\pm(\zeta_{11}+\zeta_{11}^{-1})$, then Theorem \ref{thm-sophie} says 
$$\rank_\ZZ(J(\QQ))\leq \dim_{\FF_2} \Sel^{(2)}(\QQ,J) \leq g+\dim \Cl(L)[2] = \frac{p-1}{2}=2,$$
where $J$ is the jacobian of $y^2=f(x)$. If $f(x)$ is the minimal polynomial of $\zeta_{11}+\zeta_{11}^{-1}$, then $f(x)=x^5 + x^4 - 4x^3 - 3x^2 + 3x + 1$. Below we will describe a general method to find some rational points on the jacobian, and show that $2\leq \rank_\ZZ(J)\leq 2$. Thus, the rank is $2$ and the bound is sharp. 
\end{example}

\begin{example} 
If $q=23$ and $p=11$, then $L=\QQ(\zeta_{23})^+$ is a field of degree $11$ and trivial class group. Since $2$ is a primitive root modulo $11$, if $f(x)$ be the minimal polynomial of $\pm(\zeta_{23}+\zeta_{23}^{-1})$, then Theorem \ref{thm-sophie} says 
$$\rank_\ZZ(J(\QQ))\leq \dim_{\FF_2} \Sel^{(2)}(\QQ,J) \leq g+\dim \Cl(L)[2] = \frac{p-1}{2}=5,$$
where $J$ is the jacobian of $y^2=f(x)$. If $f(x)$ is the minimal polynomial of $-(\zeta_{23}+\zeta_{23}^{-1})$, then 
$$f(x)=x^{11} - x^{10} - 10x^9 + 9x^8 + 36x^7 - 28x^6 - 56x^5 + 35x^4 + 35x^3 -    15x^2 - 6x + 1.$$ Below we will show that $4\leq \rank_\ZZ(J)\leq 5$. A full $2$-descent (via Magma) shows that the rank is, in fact, equal to $4$.
\end{example}

We finish this section by describing a method to produce points in $J$ (as in Theorem \ref{thm-sophie}, and compute the rank of the subgroup generated by these points. Let $p$ be a fixed Sophie Germain prime, let $q=2p+1$, and let $L=\QQ(\zeta_q)^+$ be
the maximal real subfield of $\QQ(\zeta_q)$, where $\zeta_q$ 
is a primitive $q$-th root of unity. The minimal polynomial
for $\zeta_q+\zeta_q^{-1}$ has constant term 1 or $-1$ according to
whether $p$ is congruent to $1$ or $3 \bmod 4$. If $f\in \ZZ[x]$ is 
a polynomial with constant term $1$, then the point $(0,1)$ will 
be on the curve
\[
C:y^2=f(x)
\]
and furthermore the factorization of $f(x)-1$ will provide points
on the jacobian $J$ of $C$, allowing us to obtain a lower bound for
the Mordell--Weil rank of $J(\QQ)$.  For this reason, we define $f$ to be the minimal polynomial
of  $\theta=(-1)^{(p-1)/2}(\zeta_q+\zeta_q^{-1})$.

A lower bound for the rank can be computed by considering the
images of the factors of $f-1$ in $L^\times/(L^\times)^2$
under the map $\delta_\QQ:J(\QQ)\to H_\QQ$ (see Section \ref{sec-2descent}). Let $y_0\in\QQ$, let $g(x)$ be an
irreducible factor of $f(x)-y_0^2$, 
and let $K$ be the splitting field of $g(x)$.  Then, over $K$, we have
a factorization
\[
g(x)=\prod_{i=1}^n (x-x_i),
\]
and the points $P_i=(x_i, y_0)$ are in $C(K)$. Under the map $\delta_K\colon J(K)\to H_K$,
as a map from $C(K)$ to $L_K^\times/(L_K^\times)^2$
we have
\[
P_i\mapsto (x_i-\theta_K)(L_K^\times)^2.
\]
If $f$ remains irreducible over $K$, then $L_K$ is simply the composite
extension of $K$ and $L_q$, and $\theta_K=\theta$. As a map from $J(K)$ to $L_K^\times/(L_K^\times)^2$, $\delta_K$ is a homomorphism of groups, and hence the divisor $P_1+P_2+\cdots +P_n$ maps to
\[
\prod_{i=1}^n (x_i-\theta_K)(L_K^\times)^2=(-1)^ng(\theta_K)(L_K^\times)^2.
\]
On the other hand, the divisor $P_1+P_2+\cdots +P_n$ can be regarded
as the base extension to $K$ of a certain divisor $D$ defined over $\QQ$, hence over $\QQ$ we have
\[
D \mapsto (-1)^ng(\theta)(L^\times)^2,
\]
via the map $F_K$ of Section \ref{sec-2descent}. Thus, for each irreducible factor $g(x)$ we obtain a point on the jacobian $J(\QQ)$ that corresponds to the divisor $D=D(g)$ defined above. Moreover, since the map $J(\QQ)/2J(\QQ)\to H_\QQ$ induced by $\delta_\QQ$ is injective (\cite[Lemma 4.1]{stoll}), in order to compute the rank of the subgroup generated by $\{D(g) \}_{g}$, it suffices to check the dimension of the (multiplicative) subgroup generated by $\{ \delta_\QQ(D(g)) \}$ in $H=H_\QQ$.  

\begin{example}\label{ex-pequal5c}
For example let $q=11$ and $p=5$, as in Example \ref{ex-pequal5b}, so $L=\QQ(\zeta_{11})^+$ and $f(x)=x^5+x^4-4x^3-3x^2+3x+1$. Then, 
$f(x)-1$ factors as
\[
x(x^2 - 3)(x^2 + x - 1).
\]
Their images in $L^\times/(L^\times)^2$ via $\delta_\QQ$ are
\[
-\theta(L^\times)^2,\,(\theta^2 - 3)(L^\times)^2,\text{ and }
(\theta^2 + \theta - 1)(L^\times)^2
\]
respectively. To obtain a lower bound for the rank of $J$ it remains only to reduce $\{-\theta,\,(\theta^2 - 3),(\theta^2 + \theta - 1)\}$ to a multiplicatively
independent subset modulo squares. Since the product of all three is
\[
-(f(\theta)-1)(L^\times)^2=(L^\times)^2
\]
we see that at most two are multiplicatively independent. On the
other hand $-\theta(\theta^2 - 3)$ is not a square in $L$, so
\[
-\theta(L^\times)^2\text{ and }(\theta^2 - 3)(L^\times)^2
\]
are multiplicatively independent give us a lower bound of $2$ as the 
rank of $J$ over $\QQ$. An upper bound of $2$ for the rank was computed in Example \ref{ex-pequal5b}, hence the rank is exactly $2$.
\end{example}

In Table \ref{tab-UpperAndLower_bounds}, we have collected the upper bound given by Theorem \ref{thm-sophie}, together with some computational data of lower bounds for the rank of the jacobian $J$ associated to the first few Sophie Germain primes, obtained using the method we have described here. It is worth pointing out that the bound given by Theorem \ref{thm-sophie} in these examples is unconditional (i.e., not dependent on the Davis--Taussky conjecture) since $q\leq 92459$. We also note that $2$ is inert in $\QQ(\zeta_p)^+$ in some cases such as $p=5,11,23,29,53,83,131,173$ but not in others such as $p=41,89,113$.

\begin{table}[h!]
\begin{center}
\renewcommand{\arraystretch}{1}
\begin{tabular}{|r||lllllllllll|}\hline
$p$& 5& 11&23&29&41&53&83&89&113&131&173\\
\hline
upper& 2 & 5&11&14&20&26&41&44&56&65&86\\
\hline
lower& 2 & 4&6&4&4&4&10&6&4&10&4\\
\hline
\end{tabular}
\end{center}
\caption{Upper and lower bounds for the rank of $J_q(\QQ)$ when $p=(q-1)/2$ is Sophie Germain.}  \label{tab-UpperAndLower_bounds}
\end{table}

\renewcommand{\arraystretch}{1}
\bibliography{Alvaro-Harris-Erik}{}
\bibliographystyle{plain}
\end{document}